\documentclass[a4paper,10pt]{amsart}

\usepackage{amssymb}
\usepackage{latexsym}
\usepackage{amsfonts}
\usepackage{amsmath}
\usepackage{amssymb}
\usepackage{mathtools}
\usepackage{amsxtra}
\usepackage{amscd}
\usepackage[english]{babel}
\usepackage{graphicx}
\usepackage{amscd}
\usepackage{color}
\usepackage{subfigure}

\DeclareMathOperator{\idx}{\mathrm{i}}

\newcommand{\R}{\mathbb{R}} 
\newcommand{\N}{\mathbb{N}}
\newcommand{\Z}{\mathbb{Z}}

\newcommand{\fix}{\mathop\mathrm{Fix}\nolimits}

\newcommand{\sign}{\mathop\mathrm{sign}\nolimits}
\newcommand{\cl}[2][]{\overline{#2}^{_{#1}}}
\newcommand{\Fr}[2][]{\mathrm{Fr}_{_{#1}}{(#2)}}

\newcommand{\dom}{\mathcal{D}} 
\newcommand{\Orb}{\mathcal{O}}
\newcommand{\D}{\partial}
\newcommand{\X}{\times} 

\newcommand{\Ker}{\mathop\mathrm{ker}\nolimits}

\newcommand{\F}{\mathcal{F}}
\newcommand{\xx}{\boldsymbol{\mathcal{X}}}
\newcommand{\Sc}{\boldsymbol{\mathcal{S}}}

\newcommand{\dif}{\mathrm{\,d}}
\DeclareMathOperator{\Pr1}{\pi_1}

\newtheorem{definition}{Definition}[section]
\newtheorem{theorem}{Theorem}[section]
\newtheorem{corollary}[theorem]{Corollary}
\newtheorem{lemma}[theorem]{Lemma}
\newtheorem{example}[theorem]{Example}
\newtheorem{proposition}[theorem]{Proposition}
\newtheorem{remark}[theorem]{Remark}

 \usepackage[top=2cm, bottom=1.5cm, outer=3.5cm, inner=3.5cm, heightrounded,%
               marginparwidth=2.6cm, marginparsep=0.5cm]{geometry}

\title[Non-$T$-resonance and multiplicity results for ODEs\ldots ]%
{About the notion of non-$T$-resonance and applications to topological multiplicity 
results for ODEs on differentiable manifolds}

\author{Luca Bisconti and Marco Spadini}
\address[L.\ Bisconti and M.\ Spadini]{Dipartimento di Matematica e
  Informatica, Universit\`a di Firenze, Via S.\ Marta 3, 50139
  Firenze, Italy} 
\begin{document}

\begin{abstract} By using topological methods,  mainly the degree of a tangent vector field, 
we establish multiplicity results for $T$-periodic solutions of parametrized $T$-periodic 
perturbations of autonomous ODEs on a differentiable manifold $M$. In order to provide insights
into the key notion of $T$-resonance, we consider the elementary situations $M=\R$ and $M=\R^2$.
So doing, we provide more comprehensive analysis of those cases and find improved conditions.
\end{abstract}
\maketitle


\section{Introduction}
The search of multiplicity results for periodic solutions is a growing area of research. 
In this paper, we consider parametrized periodic perturbations of autonomous ODEs on
differentiable manifolds and concentrate upon those results which are attainable through an 
intrinsically topological route. In doing this, we follow mainly \cite{FS98} and \cite{FPS00} 
basing our argument on the interplay between global and local results (Theorem \ref{genmu}). A 
key notion for the ``local'' part of this approach is that of \emph{ejecting set} or \emph{point}
(see Section \ref{secMul} for a definition). 
Although it is sometimes possible to prove directly that the property of being ejecting holds 
for some points, usually the most practical way is through a non-$T$-resonance condition (see, 
e.g., \cite[Ch.\ 7]{Cr94}, a similar idea can be traced back to Poincar\'e, see \cite{Mawpend} 
for an exposition). 

We point out that many other approaches to multiplicity results for ordinary differential 
equations (especially of second order) have been pursued successfully. An exhaustive list of 
those is impossible to give here; we only mention \cite{Fo12, Ma90, Pa08, Re96, Ur05} and 
references therein which we consider representative of different aspects and techniques.

The main purpose of this paper is to explore the concept of (non-) 
$T$-resonant zero of a vector field, providing an intuitive geometric grasp for it. We 
attempt to achieve this goal by considering the simplest possible situations, namely $M=\R$ 
and $M=\R^2$, for very smooth vector fields so that the implicit function theorem helps us
make sense of the local shape of the set of $T$-periodic solutions.  Indeed, in our 
investigation we come across a rather sharp description of the local structure of this set 
for scalar differential equations providing multiplicity/non-existence results (Proposition 
\ref{proNTomu}). In the two-dimensional case, we find an improved condition for an isolated 
zero of a vector field on $\R^2$ to be ejecting (Theorem \ref{condisofin}). 

Throughout this paper, we deal with the following parametrized differential equation 
\begin{equation}
\dot x=g(x)+\lambda f(t,x),\qquad\lambda \in [0,\infty ),  \label{eq0}
\end{equation}
where, throughout this section, $g\colon M\to\R^k$ and $f\colon\R\X M\to\R^k$ are continuous 
tangent vector fields on a boundaryless manifold $M\subseteq\R^k$ and $f$ is $T$-periodic in the 
first variable. Our above mentioned multiplicity result (Theorem \ref{genmu}) gives a lower bound 
on the number of $T$-periodic solutions of \eqref{eq0} for any fixed $f$ when $\lambda\geq 0$ is 
sufficiently small. The bound is based only on the properties of $g$,  Roughly speaking, we prove 
that if $M\subseteq\R^k$ is closed and $g$ has $n-1$ non-$T$-resonant zeros with local index not 
adding up to the degree of $g$, then, fixed the forcing field $f$, there exists $\lambda_*>0$ such 
that \eqref{eq0} has at least $n$ geometrically distinct $T$-periodic solutions. Recall that two 
$T$-periodic solutions $x$ and $y$ of \eqref{eq0} are said to be \emph{geometrically distinct} if 
their images do not coincide; in fact, $x$ and $y$ are not geometrically distinct if and only if 
there exists $\tau\in(0,T)$ such that $x(t+\tau)=y(t)$.

Our approach is theoretically-oriented but the kind of equations treated is common to many physical
and engineering models. For instance, under appropriate conditions, differential-algebraic equations, 
which are a very well-known modeling and simulation tool for constrained systems, can be reduced to
the form \eqref{eq0} see, e.g.\ \cite{KM}. In fact, the technique used in the proof of our main
multiplicity result, Theorem \ref{genmu}, is shared by a number different contexts. For example 
differential-algebraic equations have been considered in \cite{Bis2012,spaDAE} and constrained 
mechanical systems in \cite{FS98}.

We conclude this introduction noting that the pictures present in the paper have been produced by 
the popular interactive plotting package Gnuplot \cite{gnuplot} using data generated by simple 
software written by one of the authors (M.\ Spadini) for the purpose of analyzing the set of 
starting points for scalar and two-dimensional parametrized equations.\footnote{The version of this
software, adapted for scalar equations, is available through its author web page
\texttt{http://www.dma.unifi.it/$\sim$spadini/Software/01\_Utilities/StartingPoints-1D}, 
whereas the one for two-dimensional equations is still under development.}

\section{Notions and preliminary results}\label{secNPR}

Theorem \ref{tuno} below is the main basis of the results discussed here. It is, deeply grounded 
on topological methods, mainly fixed point index theory and the degree of a tangent vector field. 
However, here, the reader does not need to know the details of the former theory which are 
completely hidden in the proof. The same is not true for the latter, though. Good references for 
it are \cite{GP,MilTop}. Here, we confine ourselves to reminding that, given a tangent vector 
field $v\colon M\to\R^k$,  here $M\subseteq\R^k$ denote a boundaryless differentiable manifold, 
and an open subset $W\subseteq M$ we say that $v$ is admissible on $W$ if $v^{-1}(0)\cap W$ is 
compact. In this case it is well-defined an integer $\deg(v,W)$, called the degree or characteristic 
of $v$ in $W$, which roughly speaking counts algebraically the number of zeros of $v$ that lie in 
$W$.%
\footnote{In the sense that when the zeros of $v$ are all non-degenerate, then the set
$v^{-1}(0)\cap W$ is finite and
\begin{equation*} 
\deg(v, W) = \sum_{{p}\in v^{-1}(0)\cap W} \sign\, \det v'(p).
\end{equation*}
It is in order to point out here that this notion of degree should not be confused with the 
(Brouwer) degree of maps between oriented manifolds, with which it morally coincides in the 
special case when $M=\R^k$. Indeed the degree of a tangent vector field satisfies all the 
classical properties of the Brouwer degree: \emph{Solution, Excision, Additivity, Homotopy 
Invariance, Normalization} etc. (see also \cite{Llo}).
}

The Excision property allows the introduction of the notion of index of an isolated zero. 
This concept will play a crucial role throughout this paper.
Given an isolated zero $p_0$ of a vector field $v$ we put $\idx(v,p_0):=\deg(v,U)$
where $U$ is an isolating neighborhood of $p_0$. Actually, if $v$ is $C^1$ at $p_0$, then
$v'(p_0)$ is an endomorphism of $T_{p_0}M$ (see e.g.\ \cite{MilTop}) and, if $v'(p_0)$
is also invertible, then $\idx(v,p_0)=\sign\det v'(p_0)$. The additivity property of
the degree implies that if $v$ is admissible in $W$ and $p_1,\ldots,p_n\in W$ are
isolated zeros of $v$, then
\[
 \deg(v,W)=\sum_{j=1}^n\idx(v,p_j) + \deg\big(v,W\setminus\{p_1,\ldots,p_n\}\big).
\]

\smallskip
Let us now recall some notions relative to the set of $T$-periodic solutions of \eqref{eq0}.

Many considerations throughout the paper depend on the following version of Ascoli's theorem.
\begin{theorem}\sloppy
Let $Y$ be a subset of $\R^k$ and $B$ a bounded equicontinuous subset of $C\big([a,b],Y\big)$. 
Then $B$ is totally bounded in $C\big([a,b],Y\big)$. In particular, if $Y$ is closed, then 
$B$ is relatively compact.
\end{theorem}

We say that $(\lambda ,x)\in[0,\infty)\X C_T(M)$ is a \emph{$T$-pair} if $x(\cdot)$ 
is a solution of \eqref{eq0} corresponding to $\lambda$. If $\lambda=0$ and $x$ is
constant, then $(\lambda ,x)$ is said to be \emph{trivial}. One may have nontrivial 
solutions even when $\lambda=0$, as it happens, for instance, when $T=2\pi$, $M=\R^2$ 
and $g(x_1,x_2)=(-x_2,x_1)$, $x=(x_1,x_2)$; clearly  any $2\pi$-periodic $f$ can be 
taken for this example.

\smallskip
Denote by $\xx$ the subset of $[0,\infty)\X C_T(M)$ of all the $T$-pairs. Well-known 
properties of differential equations imply that $\xx$ is closed. Hence, as a closed 
subset of a locally complete space, it is locally complete as well. The following 
technical remark from \cite{FS98} is noteworthy:

\begin{remark}\label{compl}
The space $\xx$ of all the $T$-pairs is locally totally bounded. Thus, being locally 
complete, $\xx$ is locally compact. Moreover, if we assume that $M$ is a complete 
manifold, then any bounded subset of $\xx$ is actually totally bounded. As a 
consequence, when the manifold $M\subseteq\R^k$ is a closed subset, closed and bounded 
sets of $T$-pairs are compact. 
\end{remark}

As in \cite{FP95} we stipulate some conventions for the sake of simplicity. Accordingly, 
we regard every space as its image in the following diagram of closed embeddings: 
\begin{equation}\label{diagram}
\begin{CD}
\left[ 0,\infty \right) \times M @>>> \left[ 0,\infty \right) \times
C_T(M) \\  @AAA  @AAA  \\
M @>>>  C_T(M).
\end{CD}
\end{equation}
In particular, $M$ will be identified with its image in $C_T(M)$ under the
embedding which associates to any $p\in M$ the map $\hat p\in C_T(M)$
constantly equal to $p$. And also, $M$ will be regarded as the slice 
$\{0\}\X M\subseteq[0,\infty)\X M$. Similarly, $C_T(M)$ will be regarded as 
$\{0\}\X C_T(M)$. 

According to these identifications, given an open subset $\Omega$ of $[0,\infty)\X C_T(M)$, 
by $\Omega \cap M$ we mean the open subset of $M$ given by all $p\in M$ such that the pair 
$(0,\hat p)$ belongs to $\Omega $. If $U$ is an open subset of $[0,\infty)\X M$, then 
$U\cap M$ represents the open set $\{p\in M:(0,p)\in U\}$. We point out that with the above 
conventions, $g^{-1}(0)$ can be viewed as the set of trivial $T$-pairs.\medskip\ 

We will heavily use the following result about the structure of the set $\xx$ of $T$-pairs of 
\eqref{eq0} that was proved in \cite{FS97}.

\begin{theorem}\label{tuno}
Let $f\colon\R\X M\to\R^k$ and $g\colon M\to\R^k$ be continuous tangent vector fields 
defined on a (boundaryless) differentiable manifold $M\subseteq\R^k$, $f$ being 
$T$-periodic in the first variable. Let $\Omega$ be an open subset of 
$[0,\infty)\X C_T(M)$, and assume that $\deg(g,\Omega \cap M)$ is well defined and 
nonzero. Then there exists a connected set $\Gamma$ of nontrivial $T$-pairs of 
\eqref{eq0} in $\Omega$ whose closure in $[0,\infty)\X C_T(M)$ intersects 
$g^{-1}(0)\cap\Omega$ and is not contained in any compact subset of $\Omega$. In 
particular, if $M$ is closed in $\R^k$ and $\Omega =[0,\infty)\X C_T(M)$, then 
$\Gamma$ is unbounded.
\end{theorem}

\begin{definition}
Let $f$ and $g$ be smooth and let $\Sc\subseteq [0,\infty)\X M$ be the topological subspace 
consisting of all pairs $(\lambda ,p)$ such that the maximal solution of the following 
Cauchy problem: 
\begin{equation}\label{cpuno}
\left\{ 
\begin{array}{l}
\dot x=g(x)+\lambda f(t,x), \\ 
x(0)=p,
\end{array}
\right. 
\end{equation}
is $T$-periodic. The elements of $\Sc$ are called \emph{starting points}. If $(\lambda,p)\in\Sc$ 
the point $p$ will be informally referred to as an \emph{initial point} of a $T$-periodic solution.
\end{definition}

An important fact that follows from the proof of Theorem \ref{tuno} in \cite{FS97} 
is the following:
\begin{remark}\label{StPo}
Let $h\colon\xx\to\Sc$ be the map which assigns to any $T$-pair $(\lambda ,x)$ the starting 
point $\big(\lambda ,x(0)\big)$. Observe that $h$ is continuous, onto and, since $f$
and $g$ are smooth, it is also one to one. Furthermore, by the continuous dependence 
on initial data, we get the continuity of $h^{-1}\colon\Sc\to\xx$. Thus $h$ maps 
$\xx\cap\Omega$ homeomorphically onto $\Sc\cap U$. A starting point $(\lambda ,p_0)$ is called 
\emph{trivial} when $\lambda=0$ and $p_0\in g^{-1}(0)$. Notice that trivial $T$-pairs 
correspond to the trivial starting points under the homeomorphism $h$. 
\end{remark}

Suppose that $f$ and $g$ are such that Cauchy problem \eqref{cpuno} admits unique maximal
solution $x(\lambda,p,\cdot)$ for any $p\in M$ and $\lambda\in[0,\infty)$. Given 
$\lambda\in[0,\infty)$ and $t\in\R$, we set $P_t^\lambda(p):=x(\lambda,p,t)$ whenever
this makes sense. Consider, in particular $t=T$. Well-known properties of differential 
equations imply that the domain $\dom$ of the map $(\lambda,p)\mapsto P_T^\lambda(p)$ is 
an open subset (possibly empty) of $[0,\infty)\X M$. In particular, for any $\lambda\geq 0$,
the slice $\dom_\lambda:=\{p\in M: (\lambda,p)\in\dom\}$ is the (open) domain of the map 
$p\mapsto P_T^\lambda(p)$. Clearly, the set $\Sc$ defined above is contained 
in $\dom$ and we have
\[
\Sc=\bigcup_{\lambda\in[0,\infty)}\fix(P_T^\lambda),
\]
where $\fix(P_T^\lambda)$ denotes the set of fixed points of the map 
$P_T^\lambda(\cdot)\colon\dom_\lambda\to M$.

\smallskip
The importance of the property highlighted in the remark will become apparent in later
sections. 
In fact, in order to illustrate the meaning of the non-$T$-resonance condition (see 
next section) we will consider explicit examples in $\R$ and $\R^2$ and plot 
numerically (a part of) the image under $h$ of the set $\Gamma$ of Theorem \ref{tuno}.
To the same end, in the case $M=\R^k$ we will also apply the implicit function theorem 
to the function $\F\colon\dom\subseteq[0,\infty)\X\R^k\to\R^k$ defined by 
\[
\F(\lambda,p):=P_T^\lambda(p)-p
\]
whose zeros are exactly the starting points of \eqref{eq0}.

\section{Ejecting sets and multiplicity}\label{secMul}
In this section we gather some of the notions which are central to our approach to
multiplicity results. The main references for this section are \cite{FS98,FPS00}.
We start with some purely topological facts and definitions.

Let $Y$ be a metric space and $X$ a subset of $[0,\infty)\X Y$. Given $\lambda\geq 0$, 
let $X_{\lambda}$ be the slice $\big\{ y\in Y:(\lambda ,y)\in X \big\}$. Moreover, 
given a topological space $S$ and two subsets $A$ and $B$, with $A\subseteq B$, 
$\cl[B]{A}$ and $\cl{A}$ will denote the closure of $A$ in $B$ and in $S$, 
respectively. Analogously, by $\Fr[B]{A}$ and by $\Fr{A}$ we refer to the boundary of 
$A$ relative to $B$ and to $S$ respectively. Finally, $\#Z$ denotes the cardinality 
of a set $Z$.

\begin{lemma}[\cite{FPS00}]\label{conn1} 
Let $Y$ be a metric space and let $X$ be a locally compact subset of $[0,\infty )\X Y$. 
Assume $K$ is a compact relatively open subset of the slice $X_0$. Then, for any 
sufficiently small open neighborhood $U$ of $K$ in $Y$, there exists a positive number 
$\delta$ such that
\[
X\cap\big( [0,\delta ]\times\mathrm{Fr}(U)\big) =\emptyset\;.
\]
\end{lemma}

\begin{definition}\label{defej}
Let $X$ be a subset of $[0,\infty )\times Y$. We say that $A\subseteq X_0$ is an 
\emph{ejecting set} (for $X$) if it is relatively open in $X_0$ and there exists a 
connected subset of $X$ which meets $A$ and is not contained in $X_0$. 
\end{definition}

If a singleton $\{p_0\}\subseteq X_0$ is an ejecting set, the point $p_0$ will be
called, with slight abuse of notation, an \emph{ejecting point}. An important 
class of ejecting points will be presented below where the notion of 
non\hbox{-}$T$\hbox{-}resonant point of \eqref{eq0} is introduced. In fact, in the
context of multiplicity results for $T$-periodic solutions, the sets $X$ and $Y$ above 
will be defined as $X=\xx$ and $Y=C_T(M)$.

We will say that a point $p_0\in g^{-1}(0)$ is $T$-resonant for \eqref{eq0} if 
$g$ is $C^1$ in a neighborhood of $p_0$ and if the linearized problem (on $T_{p_0}M$) 
\begin{equation*}
\left\{ 
\begin{array}{l}
\dot x=g^{\prime }(p_0)\,x, \\ 
x(0)=x(T)\;,
\end{array}
\right.
\end{equation*}
which corresponds to $\lambda =0$, admits nontrivial solutions.

Notice that a point $p_0\in g^{-1}(0)$ is not $T$-resonant if $g'(p_0)$ (which maps $T_{p_0}M$ 
into itself, see e.g. \cite{MilTop}) has no eigenvalues of the form $\frac{2n\pi i}T$ with 
$n\in\Z$. Thus, in particular, $p_0$ is an isolated zero of $g$ and $\idx(g,p_0)\neq 0$.

The following lemma from \cite{FS98} will play an important role. For reasons of 
future reference we reproduce its proof with some small simplifications.

\begin{lemma}\label{stretto}
Let $g\colon M\to\R^k$ be a $C^1$ tangent vector field. If $p_0\in g^{-1}(0)$ is not 
$T$-resonant, then for any sufficiently small neighborhood $V$ of the constant 
function $\hat p_0\equiv p_0$ in $C_T(M)$ there exists a real number $\delta_V>0$ such 
that $[0,\delta_V]\X\Fr{V}$ does not contain any $T$-pair of \eqref{eq0}.
\end{lemma}

\begin{proof}
Since the set $\xx$ of the $T$-pairs of \eqref{eq0} is locally compact (see Remark
\ref{compl}), there exists an open neighborhood $W$ of $\hat p_0$ in $C_T(M)$ and $\mu >0$ 
such that $\big([0,\mu]\X\cl{W}\big)\cap\xx$ is compact. 

We claim that $(\{0\}\X\cl{W})\cap\xx=\{(0,\hat p_0)\}$ when $W$ is sufficiently small. 
To prove this claim, assume by contradiction,that there exists a sequence of nontrivial 
$T$-pairs of the form $\{(0,x_n)\}_{n\in\N}$ converging to $(0,\hat p_0)$.  Thus 
$\big| x_n(t)-p_0\big|\to 0$ uniformly in $t$ as $n\to\infty $. In particular, 
$\lim_{n\to\infty}x_n(0)=p_0$. Put 
\begin{equation*}
p_n:=x_n(0)\quad\text{and}\quad u_n:=\frac{p_n-p_0}{\left| p_n-p_0\right| }.
\end{equation*}
We can assume that $u_n\to u\in T_pM$. Let $P_t\colon M\to M$ be the Poincar\'e 
$t$-translation operator associated with the equation $\dot x=g(x)$. Since $g$ is 
$C^1$, the map $P_t(\cdot )$ is differentiable. Define $\Phi\colon M\to\R^k$ as 
$\Phi(\xi )=\xi-P_T(\xi)$. Then, $\Phi$ is differentiable as well and 
$\Phi(p_n)=\Phi(p_0)=0$. Thus, 
\begin{equation*}
\Phi'(p_0)u=\lim_{n\to\infty}\frac{\Phi(p_n)-\Phi(p_0)}{\left| p_n-p_0\right| }=0.
\end{equation*}
On the other hand $\Phi'(p_0):T_{p_0}M\to\R^k$ is given by $\Phi'(p_0)v=v-P_T'(p_0)v$
for all $v\in T_{p_0}M$.
As it is well known, the map $\alpha\colon t\mapsto P_t'(p_0)u$ satisfies the Cauchy 
problem 
\begin{equation}\label{CPalpha}
\left\{ 
\begin{array}{l}
\dot \alpha (t)=g^{\prime }(p_0)\alpha (t) \\ 
\alpha (0)=u\;.
\end{array}
\right.
\end{equation}
Since $p_0$ is not $T$-resonant, $\Phi'(p_0)u=\alpha(0)-\alpha(T)\neq 0$. This 
contradiction proves the claim.

Let us now complete the proof. Take $W$ satisfying the above properties. We have that 
$\{(0,\hat p_0)\}=\cl{W}\cap\xx_0$ is compact and it is, being isolated, relatively open 
in $\xx_0$. The assertion now follows from Lemma \ref{conn1}.
\end{proof}

Lemma \ref{stretto} might look a little trivial at first. In fact, if $M=\R^k$ and 
$f$ is sufficiently regular, its assertion can be deduced by the implicit function 
theorem applied to the equation $\F(\lambda,p)=0$, where $\F$ is as in Section 
\ref{secNPR}. However, Lemma \ref{stretto} holds true on any differentiable manifold
and even when $f$ is only continuous.

\begin{remark}\label{remEj}
Let $p_0$ be a non-$T$-resonant zero of $g$ and let $V$ be a neighborhood of $p_0$ in 
$C_T(M)$ as in Lemma \ref{stretto}. Put $\Omega=[0,\infty)\times V$. It is not 
difficult to show that $\deg(g,\Omega\cap M)=\idx(g,p_0)\neq 0$.  Theorem \ref{tuno}
shows that there exists a connected set $\Gamma\subseteq\Omega$ of nontrivial 
$T$-pairs of \eqref{eq0} in whose closure in $[0,\infty)\X C_T(M)$ intersects 
$p_0$ and is not contained in any compact subset of $\Omega$. Thus, a combination
of Theorem \ref{tuno} and Lemma \ref{stretto} shows that $p_0$ is an isolated point of 
$\xx_0$ which is ejecting. It also imply that, if $\delta_V>0$ is as in Lemma 
\ref{stretto}, then $\Pr1(\Gamma)\supseteq [0,\delta_V]$. Here $\Pr1$ denotes the 
projection of $[0,\infty)\X C_T(M)$ onto the first factor.
\end{remark}

Our approach to multiplicity of $T$-periodic solutions is based on the following 
abstract multiplicity result from \cite{FPS00}:

\begin{theorem}[\cite{FPS00}]
\label{conn2} Let $Y$ be a metric space and let $X$ be a locally compact
subset of $[0,\infty )\times Y$. Assume that $X_0$ contains $n$ pairwise
disjoint ejecting subsets, $n-1$ of which are compact.
Then there exists $\lambda _{*}>0$ such that $\# X_\lambda \geq n$ for any
$\lambda \in [0,\lambda _{*})$.
\end{theorem}

We are now in a position to state our main multiplicity result
(compare \cite[Th.\ 3.7]{FPS00} and \cite[Th.\ 4.6]{S03}).

\begin{theorem}\label{genmu}
Let $g\colon M\to\R^k$ and $f\colon\R\X M\to\R^k$ be tangent to the closed 
boundaryless submanifold $M$ of $\R^k$. Assume that $g$ is $C^1$, $g^{-1}(0)$ is 
compact and $f$ is $T$\hbox{-}periodic in $t$. Assume also that
\begin{enumerate}\renewcommand{\theenumi}{\roman{enumi}}
\item\label{cond1} There are $n-1$ non\hbox{-}$T$\hbox{-}resonant zeros of $g$, 
$p_1,\ldots,p_{n-1}$ such that
\begin{equation}\label{eqgrad}
\sum_{i=1}^{n-1}\idx(g,p_i) \neq \deg(g,M) ;
\end{equation}
\item\label{cond2} The unperturbed equation
\begin{equation}\label{unperturbed}
\dot x =g(x)
\end{equation}
does not admit unbounded connected sets of $T$\hbox{-}periodic solutions in $C_T(M)$.
\end{enumerate}
Then, for $\lambda >0$ sufficiently small, the Equation \eqref{eq0} admits at least 
$n$ geometrically distinct $T$\hbox{-}periodic solutions.
\end{theorem}

\begin{example}\label{exNTse}
Take $M=\R$ and let $g(x)=\frac{x+x^2}{1+x^2}$ and $f(t,x)=\sin t$ with $T=2\pi$. We have two 
non-$T$-resonant zeros, $p_0=0$ and $p_1=-1$. This implies the existence of two geometrically 
independent $T$-periodic solutions for small values of $\lambda\geq 0$. Figure \ref{figNTse} 
shows the set of starting points near $p_0$ and $p_1$.
\end{example}
\begin{figure}[h!]
 \includegraphics[width=0.45\linewidth]{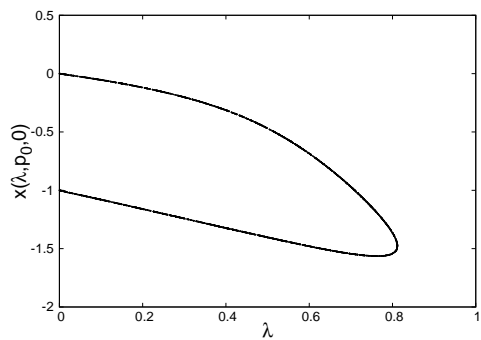}
 \caption{The set of starting points of Example \ref{exNTse}.}\label{figNTse}
\end{figure}

Note that, when $M$ is a compact boundaryless manifold condition (\ref{cond2}) of Theorem 
\ref{genmu} holds automatically. Also, in this case, by the Poincar\'e\hbox{-}Hopf 
theorem one gets $\deg(g,M)=\chi(M)$, so that \eqref{eqgrad} becomes
\[
\sum_{i=1}^{n}\idx(g,p_i) \neq\chi (M).
\]
Thus we have the following consequence of Theorem \ref{genmu}:
\begin{corollary}\label{cgenmu}
 Let $M\subseteq\R^k$ be a compact boundaryless manifold, and let $f$ and $g$ be as in
 Theorem \ref{genmu}. Assume that there are $n$ non\hbox{-}$T$\hbox{-}resonant zeros 
 of $g$, $p_1,\ldots,p_{n-1}$ such that
\[
 \sum_{i=1}^{n-1}\idx(g,p_i) \neq \chi(M).
\]
Then, for $\lambda >0$ sufficiently small, the Equation \eqref{eq0} admits at least 
$n$ geometrically distinct $T$\hbox{-}periodic solutions.
\end{corollary}

\begin{example}
Take $M=S^2\subseteq\R^3$ and $T=2\pi$. Let $g\colon M\to\R^3$ be given by 
\[
g(x,y,z)=\begin{cases}
\big(-2yz^3,2xz^3,0\big) & \text{when $z\geq 0$},\\
\big(xz^3,-yz^3,(y^2-x^2)z^2\big) & \text{when $z\leq 0$}.
\end{cases}
\]
Clearly, $g$ is $C^1$ and the points $p_1:=(0,0,1)$ and $p_2=(0,0,-1)$ are 
non-$T$-resonant zeros. with $\idx(g,p_1)=1$ and $\idx(g,p_2)=-1$. So that
$2=\chi(S^2)\neq\idx(g,p_1)+\idx(g,p_2)=0$. Therefore, for any $2\pi$-periodic
forcing term $f$ Corollary \ref{cgenmu} yields, 3 geometrically distinct 
$2\pi$-periodic solutions of Equation \eqref{eq0} when $\lambda>0$ is sufficiently 
small.
\end{example}

\begin{proof}[Proof of Theorem \ref{genmu}]
Since $p_1$,\ldots,$p_{n-1}$ are non-$T$-resonant, we can find neighborhoods
$V_1$,\ldots,$V_{n-1}$ such that
\[
\cl{V_i}\cap g^{-1}(0)=\{p_i\}\quad\text{for $i=1,\ldots,n-1$}
\]
By the definition of index of an isolated zero and the excision property, 
$\deg(g,V_i)=\mathop{\rm i}(g,p_i)$, for $i=1,\ldots,n-1$. Define
\[
V_0=M\setminus\bigcup_{i=1}^{n-1}\overline{V_i}.
\]
The additivity property of the degree yields
\[
\deg(g,V_0)=\deg(g,M)-\sum_{i=1}^{n-1}\mathop{\rm i} (p_i,g)\neq 0.
\]
Define
\[
\Omega=[0,\infty)\times C_T(V_0)\subset[0,\infty)\times C_T(M).
\]
Theorem \ref{tuno} implies that there exists an unbounded connected set $\Gamma$ 
of nontrivial $T$-pairs of \eqref{eq0} in $\Omega$ whose closure in the space
$[0,\infty)\X C_T(M)$ intersects $g^{-1}(0)\cap\Omega$. By assumption (\ref{cond2})
it follows that $g^{-1}(0)\cap V_0$ is an ejecting set of the set of $T$-pairs 
for \eqref{eq0}. Remark \ref{remEj} shows that $p_1,\ldots,p_{n-1}$ are ejecting
sets (clearly compact).  Thus, by Theorem \ref{conn2}, when $\lambda>0$ is 
sufficiently small there are $n$ different $T$\hbox{-}pairs, whence $n$ different 
$T$\hbox{-}periodic solutions. That these are geometrically distinct follows from
Lemma \ref{stretto}, for possibly smaller values of $\lambda$. 
\end{proof}

\section{$T$-resonance revisited}
Going over the first part of the proof of Lemma \ref{stretto} we see that what we actually 
prove is that if $p$ is not an accumulation point of $\Gamma\cap\big(\{0\}\X M\big)$.
This means that $\Gamma$ does not intersect $\{0\}\X M$ infinitely many times in a
neighborhood of $\{0\}\X\{p\}$, yet many pathological scenarios are possible. 
Something more can be said with the aid of the implicit function theorem if we
assume some more regularity of $f$. In fact, we have to assume that $M=\R^n$ (this
is harmless since we are going to make a local analysis) and allow $\lambda$ to
vary in $\R$ (a whole neighborhood of $0$) instead than in $[0,\infty)$.

Take $M=\R^n$ and assume that $f$ and $g$ are $C^1$. As in Section \ref{secNPR}, 
we set 
\[
\F(\lambda,p)=x(\lambda,p,T)-p
\]
where $x(\lambda,p,\cdot)$ denotes the 
unique maximal solution of \eqref{cpuno} and $\lambda\in\R$. Clearly, a pair 
$(\lambda,p)\in[0,\infty)\X M$ is a starting point (see Remark \ref{StPo}) if and
only if $\F(\lambda,p)=0$ and, also, starting points correspond to $T$-pairs when
$f$ and $g$ are $C^1$.

Let $p_0\in M$ be a non-$T$-resonant zero of $g$. Then, simple computations show 
that, for all $v\in\R^n=T_{p_0}M$,
\begin{gather}
\D_1\F(0,p_0) = e^{Tg'(p_0)}\int_0^Te^{-sg'(p_0)}f(s,p_0)\dif{s};\label{Dl}\\ 
\D_2\F(0,p_0)v = \left(e^{Tg'(p_0)}-I\right)v.\label{Dp}
\end{gather}
Thus, since $p_0$ is non-$T$-resonant, $\D_2\F(0,p_0)$ is nonsingular and, by the implicit 
function theorem, there exists a function $p:\lambda\mapsto p(\lambda)$, defined in a 
neighborhood $U$ of $0$, such that $p(0)=0$ and $\F\big(\lambda,p(\lambda)\big)\equiv 0$,
so that all pairs $\big(\lambda,p(\lambda)\big)$, $\lambda\in U$, are starting points.
Clearly,
\[
 p'(0)=-\left(e^{Tg'(p_0)}-I\right)^{-1}\int_0^Te^{(T-s)g'(p_0)}f(s,p_0)\dif{s}
\]
is well-defined.
This means, geometrically, that when $p_0$ is non-$T$-resonant $\Gamma$ crosses transversally 
$\{0\}\X M$ at $p_0$. Conversely, however, we cannot conclude that this does not happen when 
$p_0$ is $T$-resonant (we shall see this in Example \ref{exnasty}). To investigate this case 
we set $M=\R$. Later, we will consider the case $M=\R^2$ and, following the lead of Lemma
\ref{stretto}, seek a second-order sufficient condition for an isolated zero with nonzero 
index to be ejecting. 

\begin{figure}[h!]
 \begin{tabular}{c@{\hskip 1cm}c}
  \subfigure[$\lambda\in[0,0.8)$]{\includegraphics[height=0.32\linewidth]{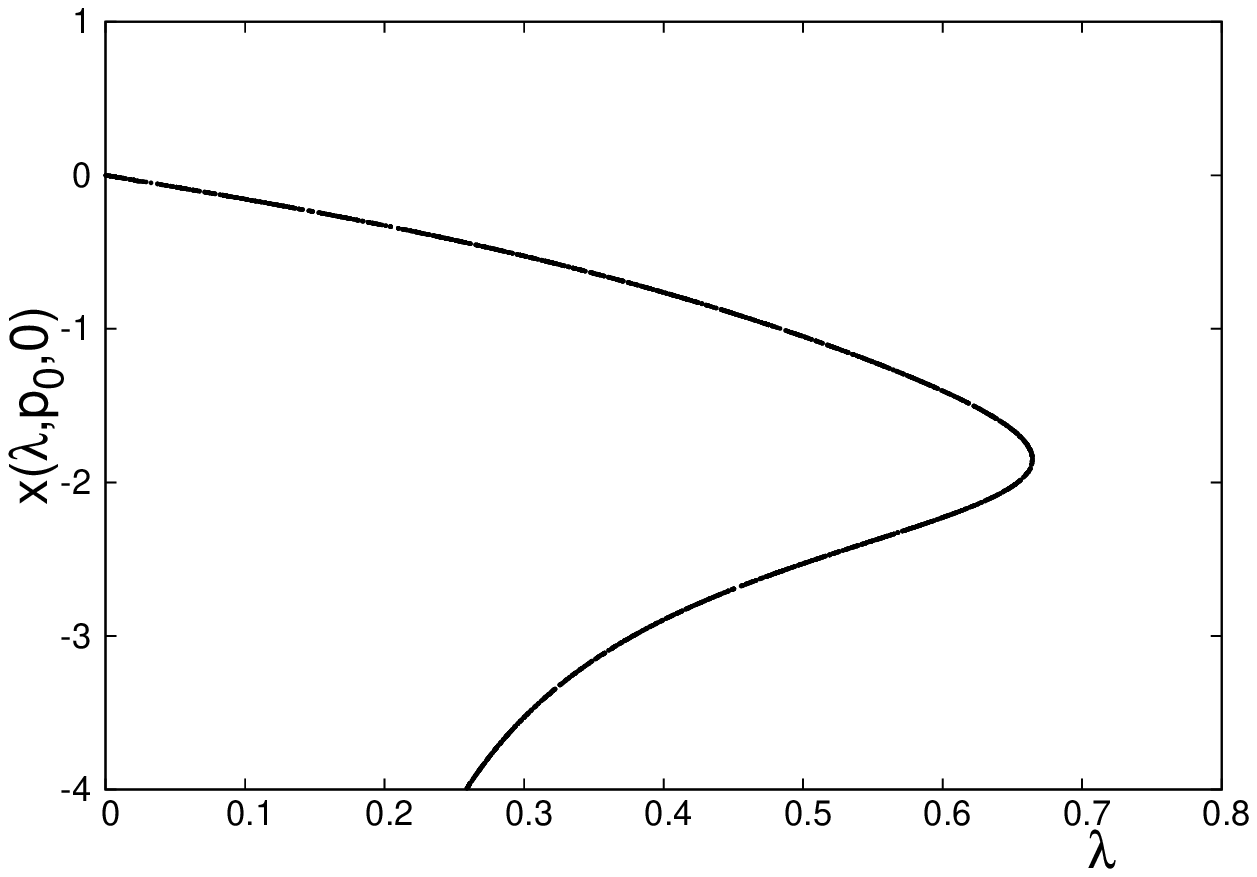}\label{figsimpA}}
  &
  \subfigure[$\lambda\in(-0.3,0.3)$]{\includegraphics[height=0.32\linewidth,width=0.32\linewidth,%
                         ]{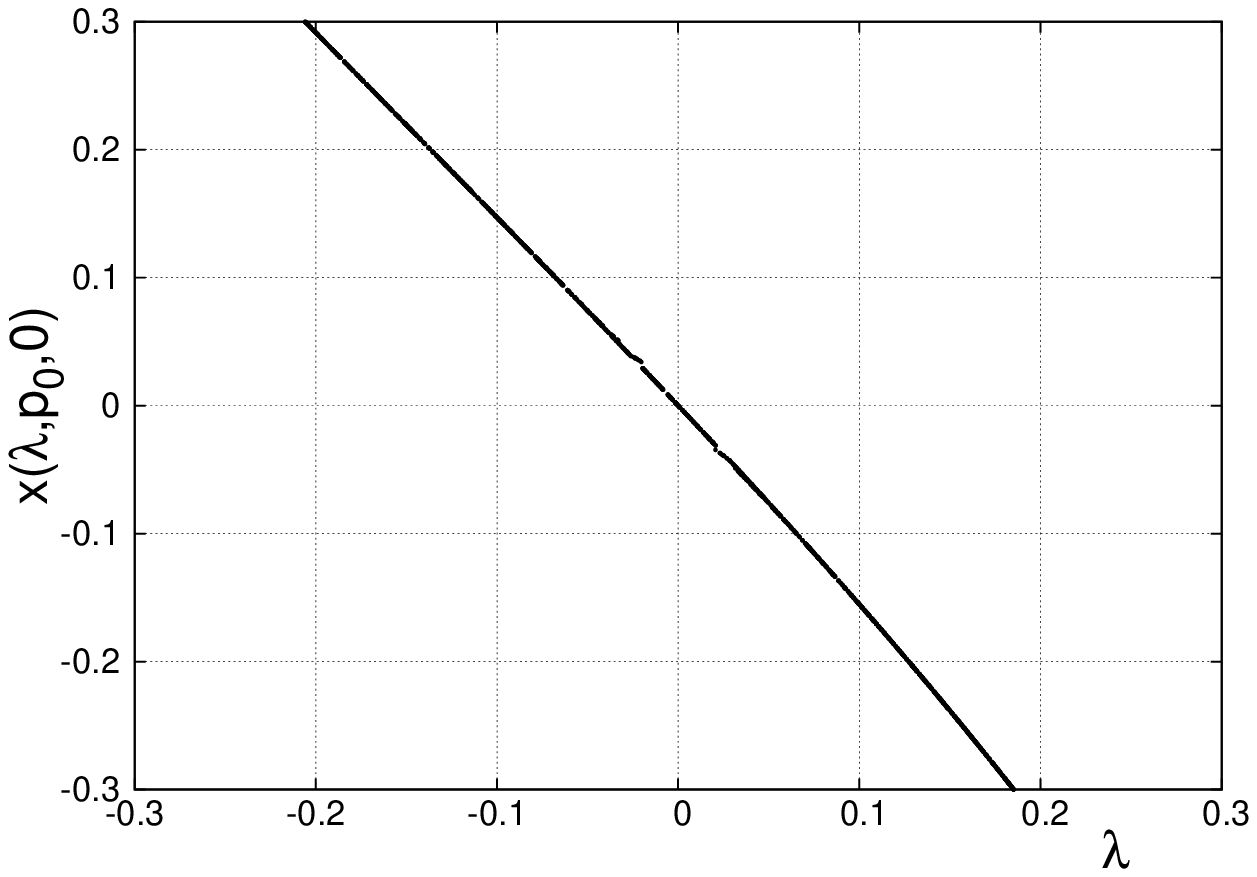}\label{figsimpB}}
 \end{tabular}
\caption{The set of starting points for Example \ref{exsimp}.}
\end{figure}
\begin{example}\label{exsimp}
 Let $M=\R$ and $g(x)=\frac{x}{1+x^2}$ and $f(t,x)=1+\cos(x+t)$ with $T=2\pi$. Clearly $g$
 has exactly one zero $p_0=0$ which is non-$T$-resonant.  Thus the implicit function theorem 
 yields exists a function $p:\lambda\mapsto p(\lambda)$, defined in a neighborhood of $0$, such 
 that $p(0)=0$ and the pairs $\big(\lambda,p(\lambda)\big)$ are starting points for \eqref{eq0}
 for small values of $|\lambda|$. Figure \ref{figsimpA} shows a portion of the set $\Sc$ of
 starting points for the above choices of $g$ and $f$. A simple computation yields
 $p'(0)=-3/2$ as it is illustrated by Figure \ref{figsimpB}.
\end{example}

\subsection{The case $M=\R$.}
In general, Remark \ref{remEj} shows that if $p_0\in M$ is a non-$T$-resonant zero of 
$g$ then the index of $p_0$ is nonzero and $p_0$ is ejecting. When $M=\R$ all isolated
zeros of $g$ with nonzero index are necessarily ejecting, as shown below.
\begin{lemma}\label{Rej}
 Let $M=\R$ and let $p_0$ be an isolated zero of $g$ such that $\idx(g,p_0)\neq 0$.
 Then $p_0$ is an ejecting point for the set $\xx$ of $T$-pairs of \eqref{eq0}. 
 Consequently, $p_0$ is an ejecting point for the set $\Sc$ of starting points of 
 \eqref{eq0}.
\end{lemma}
\begin{proof}[Sketch of the proof.]
 As in the proof of Lemma \ref{stretto}, since the set $\xx$ of the $T$-pairs of \eqref{eq0} 
 is locally compact, there exists an open neighborhood $W$ of $p_0$ in $C_T(\R)$ and 
 $\mu >0$ such that $\xx\cap\big([0,\mu]\times\cl{W}\big)$ is compact. 
 
 Since $T$-periodic solutions of an autonomous equation on $\R$ are necessarily constant
 and $p_0$ is an isolated zero, we see that $(\{0\}\X\cl{W})\cap\xx=\{(0,p_0)\}$. The proof
 is concluded in the same way as the proof of Lemma \ref{stretto}.
\end{proof}

Let $p_0\in\R$ be an isolated zero of $g$. Assume that $p_0$ is $T$-resonant,
thus $e^{Tg'(p_0)}=1$, and, since $M=\R$, $g'(p_0)=0$. In particular $\D_2\F(0,p_0)=0$ and
$\D_1\F(0,p_0)=\int_0^Tf(s,p_0)\dif{s}$. If we suppose that $\D_1\F(0,p_0)\neq 0$, the 
implicit function theorem yields a function $\lambda:p\mapsto\lambda(p)$, defined in a 
neighborhood $V$ of $p_0$, such that $\lambda(p_0)=0$ and $\F\big(\lambda(p),p\big)\equiv 0$ 
for $p\in V$. Thus, in a neighborhood of $p_0$, the set $\Gamma$ of Theorem \ref{tuno} is 
the graph of the map $\lambda:p\mapsto\big(\lambda(p),p\big)$, where $\big(\lambda(p),p\big)$ 
is regarded as a $T$-pair according to diagram \eqref{diagram}. Geometrically, since 
\[
\lambda'(p_0)= -\big(\D_1\F(0,p_0)\big)^{-1}\D_2\F(0,p_0) = 
              -\frac{e^{Tg'(p_0)}-1}{\int_0^Tf(s,p_0)\dif{s}} = 0,
\] 
this means that $\Gamma$ is tangent to $\{0\}\X\R$ at $p_0$.

\begin{example}\label{extang}
Let $g(x)=\frac{x^3}{1+x^2}$ and $f(t,x)=1+\sin t$ with $T=2\pi$. We have one $T$-resonant
zero $p_0=0$, but $\int_0^T f(s,p_0)\dif{s}\neq 0$. Thus the set $\Sc$ of starting points about $0$
is tangent to the axis $\lambda=0$, as shown in Figure \ref{figtang}.
\end{example}
\begin{figure}[h!]
\includegraphics[width=0.45\linewidth]{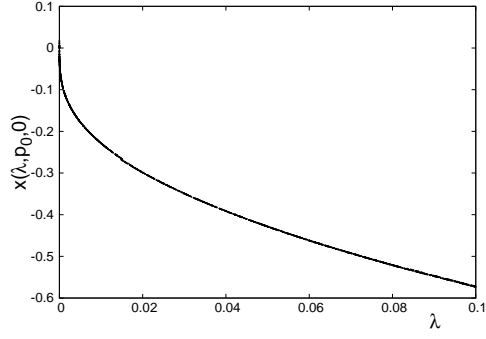}
\caption{The set of starting points of Example \ref{extang}.}\label{figtang}
\end{figure}

Notice that 
nothing can be deduced through the implicit function theorem if we have both $g'(p_0)=0$ and 
$\int_0^Tf(s,p_0)\dif{s}=0$. An example in this sense is provided below.
\begin{example}\label{exnasty}
Let $g(x)=\frac{x^3+x^2}{1+x^2}$ and $f(t,x)=\sin (t+x)$, $T=2\pi$. We have one $T$-resonant
 $p_1=0$ and one non-$T$-resonant zero $p_2=-1$, respectively. Since 
$\int_0^T f(s,p_1)\dif{s}=0$. So the shape of the set $\Sc$ of starting points about $p_1$ cannot 
be determine by the method above. Figure \ref{fignastyA} shows a portion of the set of starting 
points for \eqref{eq0} when $g$ and $f$ are selected as above. Allowing $\lambda$ in a complete
neighborhood of $0$ lets us understand the structure (see Figure \ref{fignastyB}).
\end{example}
\begin{figure}[h!]
\begin{tabular}{c@{\hskip 1cm}c}
\subfigure[$\lambda\in [0,1/2)$]{\includegraphics[width=0.45\linewidth]{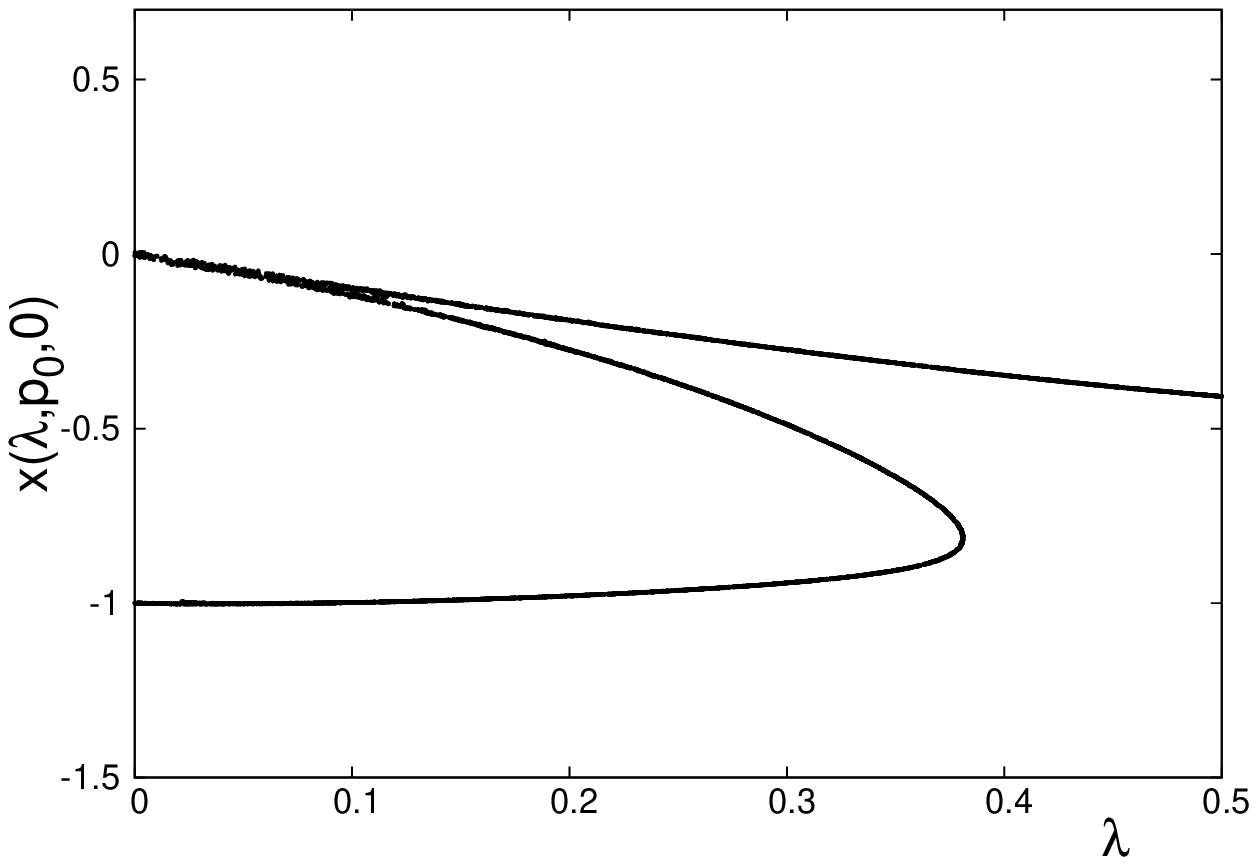}\label{fignastyA}} &
\subfigure[$\lambda\in (-1/2,1/2)$]{\includegraphics[width=0.45\linewidth]{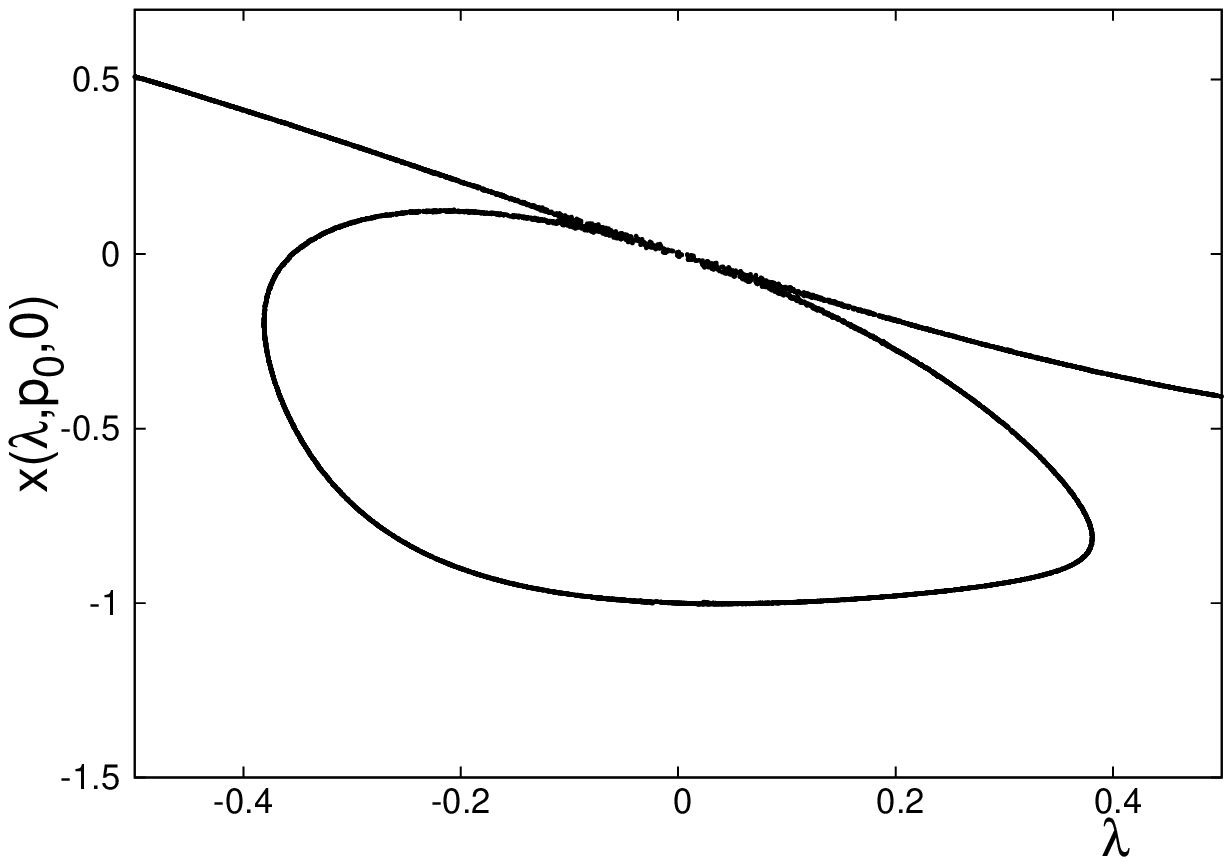}\label{fignastyB}}
\end{tabular}
\caption{The set of starting points of Example \ref{exnasty}.}\label{fignasty}
\end{figure}

We get more information about the shape of $\Gamma$ by computing the second derivative of
the function $\lambda$.

Let us stay with the assumption that $p_0$ is $T$-resonant, so that we have $g'(p_0)=0$, and 
assume in addition that $f$ and $g$ are $C^2$ functions. As above, we suppose that 
$\int_0^Tf(s,p_0)\dif{s}$ is not zero. Recall that for a $C^2$ function 
$\varphi\colon D\subseteq\R\to\R$, $D$ open, we have
\[
 \varphi''(\xi)=\lim_{h\to 0}\frac{\varphi(\xi+hv)-2\varphi(\xi)+\varphi(\xi-hv)}{h^2},
\]
for any $\xi\in D$. Then we have
\begin{equation}\label{fd22}
\begin{split}
 \D_{22}x(0,p_0,&t)= \\
= & \lim_{h\to 0} \frac{1}{h^2}\int_0^t
      \Big[g\big(x(0,p_0+h,s)\big)-2g\big(x(0,p_0,s)\big)+g\big(x(0,p_0-h,s)\big)\Big]\dif{s}\\
= & \int_0^t\left[\lim_{h\to 0^+}\frac{ g'(p_0)\D_{22}x(0,p_0,t)h^2
                     +g''(p_0)[\D_{2}x(0,p_0,s)h]^2+o(h^2)}{h^2}\right]\dif{s},
\end{split}
\end{equation}
and, since $g'(p_0)=0$, we get
\[
 \D_{22}x(0,p_0,t)= \int_0^t g''(p_0)[\D_{2}x(0,p_0,s)]^2\dif{s}
\]
Thus, setting $\alpha(t)=\D_{2}x(0,p_0,t)$ and $\beta=\D_{22}x(0,p_0,t)$, we have that $\beta$
solves the following Cauchy problem:
\[
 \left\{ 
  \begin{array}{l}
   \dot\beta(t)=g''(p_0)\alpha(t)^2,\\
   \beta(0)=0.
  \end{array}
 \right. 
\]
It is well-known that $\alpha$ solves \eqref{CPalpha} with $u=1$. Since $g'(p_0)=0$, we get 
$\alpha(t)\equiv 1$, so that $\D_{22}\F(0,p_0)=\D_{22}x(0,p_0,t)=Tg''(p_0)$. In order to 
determine $\lambda''(p_0)$ we differentiate twice the identity $\F\big(\lambda(p),p\big)=0$
at $p_0$:
\[
 \D_{12}\F(0,p_0)\lambda'(p_0)+\D_{1}\F(0,p_0)\lambda''(p_0)+\D_{22}\F(0,p_0)=0.
\]
Recalling again that $p_0$ is $T$-resonant we have $\lambda'(p_0)=0$. The above relation, 
along with \eqref{Dl}, yields
\begin{equation}\label{secderl}
 \lambda''(p_0)=-\frac{g''(p_0)}{\frac{1}{T}\int_0^Tf(s,p_0)\dif{s}}.
\end{equation}

Suppose that $g''(p_0)\neq 0$. If the signs of $\int_0^Tf(s,p_0)\dif{s}$ and $g''(p_0)$
disagree, then $\lambda''(p_0)>0$. This implies that $\lambda$ is a convex function in a
neighborhood of $p_0$. This fact has the following interesting consequence:

\begin{proposition}\label{proNTomu}
 Let  $M=\R$ and let $p_0$ be an isolated $T$-resonant zero of $g$. Assume that $f$ and $g$ 
are $C^2$ functions. Then,
\begin{enumerate}\renewcommand{\theenumi}{\roman{enumi}}
\item\label{r1} Let $g''(p_0)\int_0^Tf(s,p_0)\dif{s}>0$. Then, given $\varepsilon>0$, there exists 
$\lambda_*>0$ such that Equation \eqref{eq0} has no $T$-periodic solutions that lie in 
$[p_0-\varepsilon,p_0+\varepsilon]$ for $\lambda\in[0,\lambda_*]$;
\item\label{r2} Let $g''(p_0)\int_0^Tf(s,p_0)\dif{s}< 0$. Assume that $f$ has a separated variable 
form, that is there exist functions $\varphi$ and $h$ such that $f(t,x)=\varphi(t)h(x)$ 
for all $(t,x)\in\R\X\R$. Suppose also that $T$ is the minimal period of $\varphi$. 
Then, given $\varepsilon>0$, there exists $\lambda_*>0$ such that Equation \eqref{eq0} has 
at least two geometrically distinct $T$-periodic solutions that lie in 
$[p_0-\varepsilon,p_0+\varepsilon]$ for $\lambda\in[0,\lambda_*]$.
\item\label{r3} If $g''(p_0)\int_0^Tf(s,p_0)\dif{s}\neq 0$ then $\idx(g,p_0)= 0$;
\end{enumerate}
\end{proposition}
\begin{proof}
For the sake of simplicity we shall assume throughout this proof that $p_0=0$.

\smallskip 
\noindent\textbf{(\ref{r1})} The argument above shows that there exists a function 
$\lambda:p\mapsto\lambda(p)$, defined in a neighborhood of $p_0=0$, such that $\lambda(0)=0$ 
and $\big(\lambda(p),p\big)$ is a starting point that we regard as a $T$-pair according to 
diagram \eqref{diagram}. As already observed, the $T$-resonance of $0$ implies $g'(0)=0$ 
so that we have $\lambda'(0)=0$. Also, Formula \eqref{secderl} yields $\lambda''(0)<0$. 
Since $\lambda(\cdot)$ is a $C^2$ function we can write
\begin{equation}\label{flam}
 \lambda(p)=\frac{1}{2}\lambda''(0)p^2+\sigma\big(|p|\big)p^2,
\end{equation}
where $\sigma$ is a continuous function with $\sigma(0)=0$. Fix any $\bar\lambda>0$. We have that, 
for a sufficiently small $\varepsilon>0$, equation $\lambda(p)=\bar\lambda$ has no solution in the 
interval $[-\varepsilon,\varepsilon]$. In fact, if $p_{\bar\lambda}\in[-\varepsilon,\varepsilon]$ 
is such a solution, we have $p_{\bar\lambda}\neq 0$ and
\begin{equation}\label{eqcontr}
 \frac{1}{2}\lambda''(0) 
     = \frac{\bar\lambda}{p_{\bar\lambda}^2}-\sigma\big(|p_{\bar\lambda}|\big).
\end{equation}
Since $\sigma(0)=0$ the right-hand-side of  \eqref{eqcontr} is positive when 
$|p_{\bar\lambda}|<\varepsilon$ is sufficiently small, contradicting $\lambda''(0)<0$. This
proves the assertion.

\smallskip 
\noindent\textbf{(\ref{r2})} Clearly, Equation \eqref{flam} still holds but now we have $\lambda''(0)>0$. 
Since $\sigma(0)=0$, taking $\varepsilon>0$ sufficiently small we can assume that
\[
 \alpha(p):=\frac{1}{2}\lambda''(0)+\sigma\big(|p|\big)>0, \quad 
                                        \forall p\in[-\varepsilon,\varepsilon].
\]
Also, our assumption implies that $h(0)\neq 0$. Thus, restricting $\varepsilon$ if necessary, we
can assume that $h(p)\neq 0$ for all $p\in [-\varepsilon,\varepsilon]$.

We claim that the following equation:
\[
 \alpha(p)p^2=\bar\lambda
\]
admits two different solutions when $\bar\lambda$ is sufficiently small. To prove the claim put
\[
\ell=\min\left\{\max_{p\in [-\varepsilon,\varepsilon]}\{p\sqrt{\alpha(p)}\}\;,\;
                       \max_{p\in [-\varepsilon,\varepsilon]}\{-p\sqrt{\alpha(p)}\}\right\},
\]                       
which is clearly positive. The intermediate value theorem shows that the two equations
\[
 p\sqrt{\alpha(p)}=\lambda,\qquad -p\sqrt{\alpha(p)}=\lambda
\]
have both solutions (obviously different) in $[-\varepsilon,\varepsilon]$ for 
$\lambda\in[0,\sqrt{\ell}]$. Setting $\bar\lambda=\sqrt{\ell}$ we have the claim. Let us denote
these two solutions by $p_1(\lambda)$ and $p_2(\lambda)$. Clearly, these points are initial
points of $T$-periodic solutions of \eqref{eq0} which, for simplicity we denote by $\xi_\lambda$ 
and $\eta_\lambda$. That is, for all $t\in\R$, we put 
\[
\xi_\lambda(t)=x\big(\lambda,p_1(\lambda),t\big)\quad\text{and}\quad 
       \eta_\lambda(t)=x\big(\lambda,p_2(\lambda),t\big).
\]
We need to prove that, there exists $\lambda_*\in(0,\bar\lambda]$ such that, when 
$\lambda\in(0,\lambda_*]$ the points $\xi_\lambda$ and $\eta_\lambda$ are geometrically distinct.
By construction, both $p_1(\lambda)$ and $p_2(\lambda)$ tend to $p_0=0$ as $\lambda\to0$.
The continuous dependence on initial data implies that there exists $\lambda_*>0$ such that 
the images of $\xi_\lambda$ and $\eta_\lambda$ are contained in $[-\varepsilon,\varepsilon]$.
Suppose now, by contradiction, that $\xi$ and $\eta$ are not geometrically distinct. Then 
there exists $\tau\in(0,T)$ such that 
\[
\xi_\lambda(t)=x\big(\lambda,p_1(\lambda),t\big)
              =x\big(\lambda,p_2(\lambda),t+\tau\big)=\eta_\lambda(t+\tau).
\]
Differentiating this identity with respect to $t$ we get $\dot\xi_\lambda(t)=\dot\eta_\lambda(t+\tau)$ 
and, taking \eqref{eq0} into account, we get
\begin{multline*}
 g\big(\xi_\lambda(t)\big)+\lambda \varphi(t)h\big(\xi_\lambda(t)\big)=\\
           =g\big(\eta_\lambda(t+\tau)\big)+\lambda \varphi(t+\tau)h\big(\eta_\lambda(t+\tau)\big)
           =g\big(\xi_\lambda(t)\big)+\lambda \varphi(t+\tau)h\big(\xi_\lambda(t)\big).
\end{multline*}
This implies $\varphi(t)h\big(\xi\lambda(t)\big)=\varphi(t+\tau)h\big(\xi_\lambda(t)\big)$, for all 
$t$. Since $h$ is never zero on $\xi_\lambda(t)$, we have $\varphi(t)=\varphi(t+\tau)$, $t\in\R$, 
which contradicts the minimal period assumption.

\smallskip 
\noindent\textbf{(\ref{r3})} Suppose by contradiction that $\idx(g,p_0)\neq 0$. Assume first that 
$g''(p_0)$ and $\int_0^Tf(s,p_0)\dif{s}$ are nonzero and their signs agree. On one hand, proceeding as 
in (\ref{r1}) we see that there are no $T$-periodic solutions in a neighborhood of $p_0$ for $\lambda>0$ 
small. On the other hand we know that isolated zeros with nonzero index are ejecting for \eqref{eq0}. 
This yields a contradiction. Assume now that $g''(p_0)\int_0^Tf(s,p_0)\dif{s}<0$. The same argument 
applied to equation
\[
 \dot x = g(x)-\lambda f(t,x)
\]
(now $f$ is replaced by $-f$) yields a contradiction again.
\end{proof}

\begin{example}\label{ex2tang}
Take $g(x)=\frac{x^3(1+x)^2(x-1)^2}{1+x^6}$ and $f(t,x)=\sin (t)+1$, $T=2\pi$. We have three 
$T$-resonant zeros, say $p_1=0$, $p_2=1$ and $p_3=-1$,  but $\idx(g,p_1)\neq 0$. So $p_1$ is ejecting.
Also we have $\lambda''(p_1)=0$, $\lambda''(p_2)=-4$,  $\lambda''(p_3)=+4$. Thus there are 3 solutions 
for small $\lambda\geq 0$. Figure \ref{fig2tang} shows a portion of the set of starting pairs for 
\eqref{eq0} when $g$ and $f$ are selected as above.
\end{example}
\begin{figure}[h!]
\includegraphics[width=0.45\linewidth]{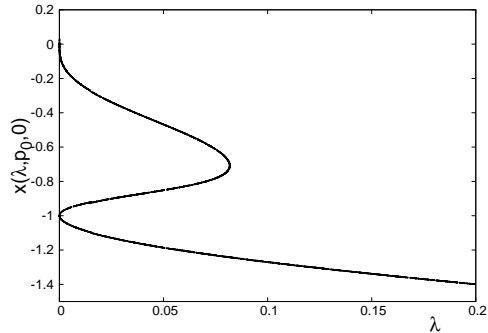}
\caption{Starting points for Example \ref{ex2tang}.}\label{fig2tang}
\end{figure}

\begin{remark}\label{remnoso}
Suppose $p_0$ is a $T$-resonant zero of $g$ and $\int_0^T f(s,p_1)\dif{s}\neq 0$. The argument of
Proposition \ref{proNTomu} shows that, if the signs of $g''(p_0)$ and $\int_0^T f(s,p_1)\dif{s}$ disagree, 
then there are two $T$-periodic solutions near $p_0$ for small $\lambda$, but none if they agree.
Figure \ref{figTNT} illustrate this phenomenon by showing a portion of the set of starting pairs for 
\eqref{eq0} when $g(x)=\frac{x^3+x^2}{1+x^2}$ and $f$ is selected in different ways.
\end{remark}
\begin{figure}[h!]
\begin{tabular}{c@{\hskip 1cm}c}
\subfigure[$f(t,x)=\sin (t)+1$ the signs agree.]{\includegraphics[width=0.45\linewidth]{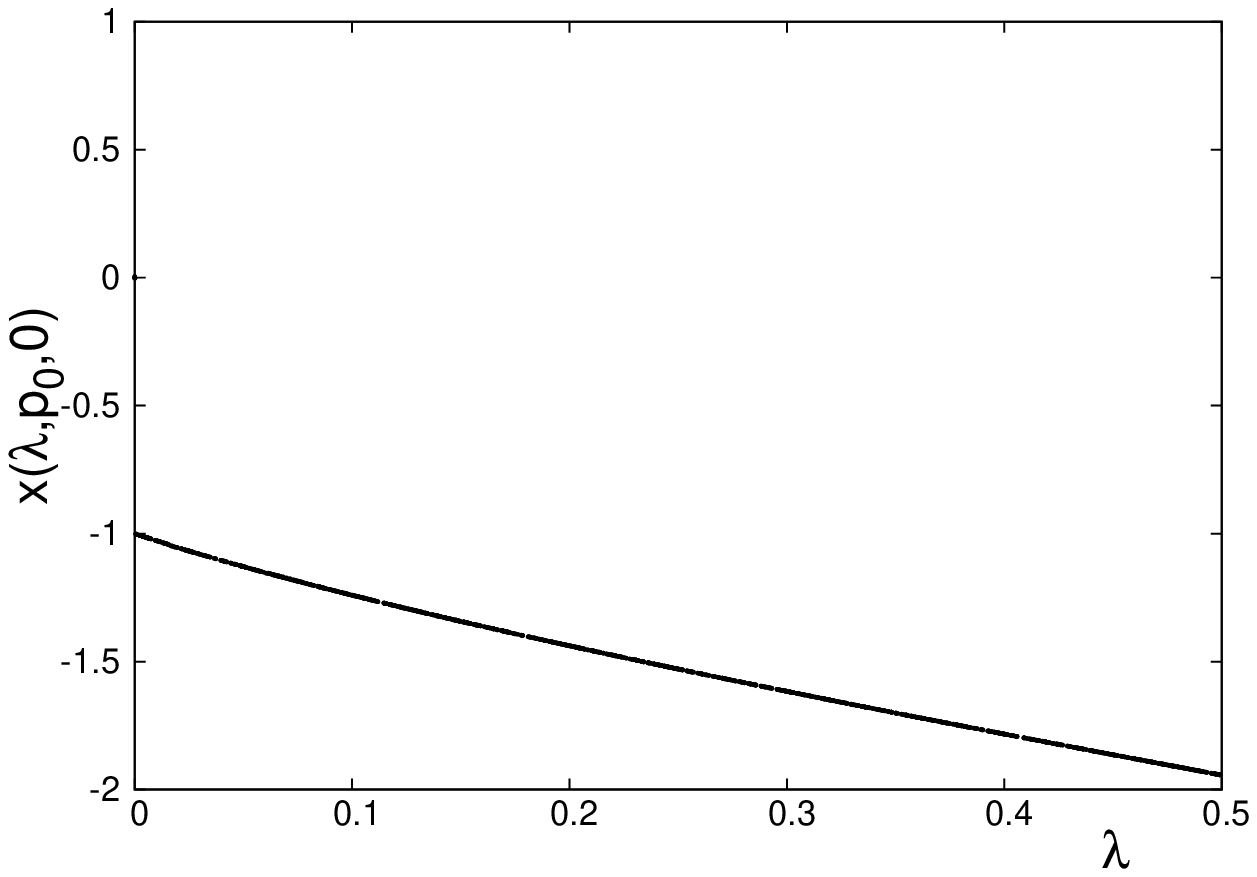}} & 
\subfigure[$f(t,x)=\sin (t)-1$ the signs disagree]{\includegraphics[width=0.45\linewidth]{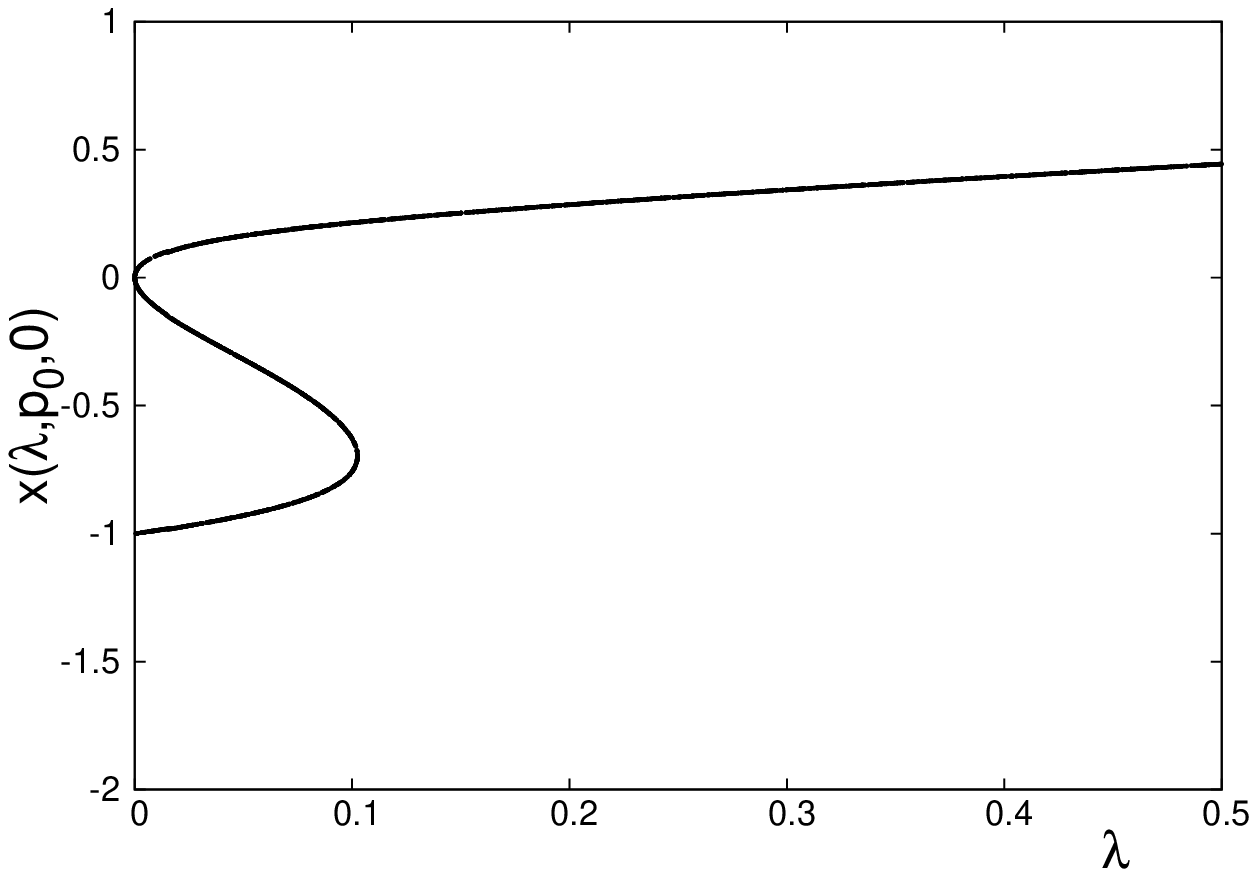}}
\end{tabular}
\caption{Two different cases for Example \ref{remnoso}.}\label{figTNT}
\end{figure}

\subsection{The case $M=\R^2$.}
When $M=\R^2$ it is no longer true that any isolated zero with nonzero index must be 
ejecting. We have, however, some structure:
\begin{lemma}\label{piattello}
Let  $M=\R^2$ and let $p_0$ be an isolated zero of $g$ that is accumulation point of 
initial points of
$T$-periodic orbits of $\dot x =g(x)$. Then, if $r$ denotes a half-line with origin in $p_0$, 
there exists a sequence $\{p_n\}_{n\in\N}\subseteq r$, with $p_n\to p_0$, such that the 
solution of $\dot x =g(x)$ starting from $p_n$ at time $t=0$ is $T$-periodic.
\end{lemma}
\begin{proof}[Sketch of the proof.]
Let $\{q_n\}_{n\in\N}$ be a sequence of initial points of $T$-periodic orbits of $\dot x =g(x)$
such that $q_n\to p_0$. Let $x_n$, $n\in\N$, be the maximal solution of the following Cauchy 
problems:
\[
 \left\{ 
  \begin{array}{l}
   \dot x=g(x),\\
   x(0)=q_n.
  \end{array}
 \right. 
\]
Then $x_n$ converges uniformly on $[0,T]$ to the function $\hat p_0(t)\equiv p_0$. Let $\Orb_n$ 
be the orbit that contains $q_n$, that is $\Orb_n=x_n\big([0,T])$. It is well-known that there
exists $z_n$ belonging to the part of plane enclosed by $\Orb_n$ such that $g(z_n)=0$ (see e.g.\
\cite[Th.\ 2, \S 5, Ch.\ 11]{HS74}). Clearly $z_n\to p_0$ but, since by assumption $p_0$ is an 
isolated zero of $g$, the only possibility is that $z_n=p_0$ eventually. This implies that the 
orbits $\Orb_n$ enclose $p_0$ when $n$ is sufficiently large. Thus, any half line with origin 
in $p_0$ necessarily meets $\Orb_n$. Let us choose a point $p_n$ in $r\cap\Orb_n$. Clearly,
since the equation is autonomous $p_n$ is an initial point of a $T$-periodic solution (like all 
the points of $\Orb_n$) and $p_n\to p_0$. 
\end{proof}

We have the following necessary condition for an isolated zero to be accumulation point of 
initial points of $T$-periodic orbits of $\dot x =g(x)$.

\begin{theorem}\label{socond}
 Let  $M=\R^2$ and let $p_0$ be an isolated zero of $g$ that is accumulation point of 
initial points of $T$-periodic orbits of $\dot x =g(x)$. Then,
\begin{equation}\label{socoeq}
 \int_0^Te^{(T-s)g'(p_0)}g''(p_0)\left[ e^{sg'(p_0)}v,e^{sg'(p_0)}v\right]\dif{s} = 0,
\end{equation}
for all $v\in\R^2$.
\end{theorem}

The proof is based on the following technical result:
\begin{lemma}\label{lemsoco}
 Let  $M=\R^2$ and let $p_0$ be an isolated zero of $g$ that is accumulation point of 
initial points of $T$-periodic orbits of $\dot x =g(x)$. Then,
\[
 \D_{22}x(0,p_0,T)[v,v] = 0,
\]
for all $v\in\R^2$.
\end{lemma}
\begin{proof}
 Let $w$ be any unit vector in $\R^2$ and consider the half-line $r_w$ generated by $w$ 
 with origin in $(0,0)$, namely the set $r_w=\{\lambda w,\lambda\geq 0\}$. Let 
 $\{q_n\}_{n\in\N}\subseteq r_w$ be a sequence as in Lemma \ref{piattello}. Then,
 \[
  \D_2x(0,p_0,T)w = \lim_{n\to\infty}\frac{x(0,q_n,T)-x(0,p_0,T)}{|q_n-p_0|}
                  =\lim_{n\to\infty}\frac{q_n-p_0}{|q_n-p_0|}=w.
 \]
 The linearity of $\D_2x(0,p_0,T)$ and the arbitrariness of $w$ show that $\D_2x(0,p_0,T)$ 
 is, in fact, the identity.
 Take now any vector $v\in\R^2$ and consider the half-line $r_v$ generated by  $v$ with 
 origin in $(0,0)$. Let $\{p_n\}_{n\in\N}\subseteq r_v$ be a sequence as in Lemma
 \ref{piattello}. Then we get
 \begin{align*}
  \D_{22}x(0,p_0,T)[v,v] 
     =& \lim_{n\to\infty}\frac{x(0,p_n,T)-x(0,p_0,T)
              -\D_2x(0,p_0,T)(p_n-p_0)}{\frac{1}{2}|p_n-p_0|^2}\\
     =& \lim_{n\to\infty}\frac{p_n-p_0-(p_n-p_0)}{\frac{1}{2}|p_n-p_0|^2}=0,
 \end{align*}
 as desired.
\end{proof}

\begin{proof}[Proof of Theorem \ref{socond}]
 We have, for $t\in [0,T]$,
 \begin{multline*}
 \D_{22}x(0,p_0,t)v= \\ =\lim_{h\to 0^+} \frac{1}{h^2}\int_0^t
      \Big[g\big(x(0,p_0+hv,s)\big)-2g\big(x(0,p_0,s)\big)+g\big(x(0,p_0-hv,s)\big)\Big]\dif{s}.
\end{multline*}
Proceeding as in \eqref{fd22} we get 
\begin{multline*}
  \D_{22}x(0,p_0,T)[v,v]=\\
              =  \int_0^t \Big[g'(p_0)\D_{22}x(0,p_0,t)[v,v]
                     +g''(p_0)[\D_{2}x(0,p_0,s)v,x(0,p_0,s)v]\Big]\dif{s},
\end{multline*}
Setting $\alpha(t)=\D_{2}x(0,p_0,t)v$ and $\beta=\D_{22}x(0,p_0,t)[v,v]$, we have that $\beta$
solves the following Cauchy problem:
\[
 \left\{ 
  \begin{array}{l}
   \dot\beta(t)=g'(p_0)\beta(t)+g''(p_0)[\alpha(t),\alpha(t)],\\
   \beta(0)=0.
  \end{array}
 \right. 
\]
As we know, $\alpha(t)=e^{tg'(p_0)}v$. Thus
\[
 \beta(t)=\int_0^Te^{(T-s)g'(p_0)}g''(p_0)\left[ e^{sg'(p_0)}v,e^{sg'(p_0)}v\right]\dif{s}.
\]
The assertion follows from Lemma \ref{lemsoco}.
\end{proof}

\begin{remark}
 Combining theorems \ref{tuno} and \ref{socond} as in the argument of Lemma \ref{stretto} we 
 see that if $p_0$ is an isolated zero of $g$ with the property that
 \begin{equation}\label{sonTres}
 \exists v\in\R^2 :
 \int_0^Te^{(T-s)g'(p_0)}g''(p_0)\left[ e^{sg'(p_0)}v,e^{sg'(p_0)}v\right]\dif{s}\neq 0,
\end{equation}
then $p_0$ must be ejecting.
\end{remark}

Condition \eqref{sonTres} is difficult to use as it is, but it gets a considerably simpler form
when $g'(p_0)$ is singular. In fact, one could prove that if $g'(p_0)$ is \emph{not} singular, 
then \eqref{sonTres} is never satisfied, in the sense that the integral in \eqref{sonTres} is
zero (the proof is based on the fact that when $g'(p_0)$ is not singular, but $p_0$ is $T$-resonant, 
$e^{tg'(p_0)}$ is necessarily a $t$-rotation matrix, so that the integrand has a particular form). 
Thus, the only interesting cases happen when $g'(p_0)$ is singular.

Let $M$, $p_0$ and $g$ be as in Theorem \ref{socond}. Suppose that $\Ker g'(p_0)\neq\{0\}$. If 
$v\in\Ker g'(p_0)$ then $e^{tg'(p_0)}v=v$, Thus, for such a $v$, equation \eqref{socoeq} reads
\[
 \int_0^Te^{(T-s)g'(p_0)}g''(p_0)[v,v]\dif{s} = 0.
\]
and condition \eqref{sonTres} becomes
\begin{equation}\label{RsonT1}
 \exists v\in\Ker g'(p_0)  :
 \int_0^Te^{(T-s)g'(p_0)}g''(p_0)[v,v]\dif{s}\neq 0.
\end{equation}
The situation is even simpler if $\Ker g'(p_0) = \R^2$. In fact, in this case $e^{tg'(p_0)}$ is
the identity for any $t$. Thus condition \eqref{sonTres} reduces to
\begin{equation}\label{RsonT2}
 \exists v\in\R^2  : g''(p_0)[v,v]\neq 0.
\end{equation}

We summarize the last discussion in the following
\begin{theorem}\label{condisofin}
Let  $M=\R^2$ and let $p_0$ be an isolated zero of $g$ with nonzero index. Then, $p_0$ is ejecting
for the set of $T$-pairs if there exists $v\in\Ker g'(p_0)$ such that
\[ 
\int_0^Te^{(T-s)g'(p_0)}g''(p_0)[v,v]\dif{s}\neq 0.
\]
If, in particular, $\Ker g'(p_0) = \R^2$ it is sufficient to find $v\in\R^2$ such that 
$g''(p_0)[v,v]\neq 0$.
\end{theorem}

\begin{example}\label{ex3d}
Let $g(x,y)=(x^3 , y+x^2)$ and $T=2\pi$. The only zero of $g$ is the origin $(0,0)$. We have
\[
 g'(0,0) =\begin{pmatrix}
           0 & 0\\
           0 & 1
          \end{pmatrix},
\]
so that $(0,0)$ is $T$-resonant but $\idx\big(g,(0,0)\big)=1$. Let us check if it is ejecting. 
Since $\det g'(0,0)=0$, we try to use condition \eqref{RsonT1}. A quick computation shows that
\[
 e^{-sg'(0,0)}=\begin{pmatrix}
           1 & 0\\
           0 & e^{-s}
          \end{pmatrix},
\qquad g''(0,0)\left[\begin{pmatrix}v_1\\ v_2\end{pmatrix},
                   \begin{pmatrix}v_1\\ v_2\end{pmatrix}\right]=
\begin{pmatrix}0\\ 2v_2^2\end{pmatrix}
\]
Taking $v=\left(\begin{smallmatrix}0\\ 1\end{smallmatrix}\right)\in\Ker g'(0,0)$, we get
\[
 \int\limits_0^T e^{-sg'(0,0)}g''(0,0)[v,v]\dif{s}=\begin{pmatrix}0 \\ 2(1-e^{2\pi})\end{pmatrix}.
\]
Thus, by Theorem \ref{condisofin}, $(0,0)$ is an ejecting zero of $g$. Figure \ref{fig3d} below
shows some points sampled numerically from the set of starting points of \eqref{eq0} when
$g$ is as above and $f(t,x,y)=\sin t +1$ ($\lambda$ is on the vertical axis).
\end{example}
\begin{figure}[h!]
\includegraphics[width=0.34\linewidth,angle=-90]{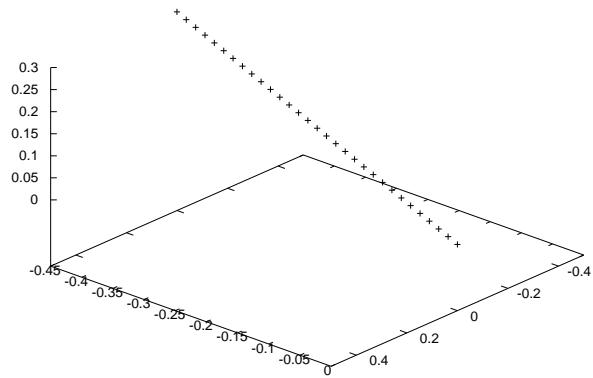}
\caption{The set of starting points of Example \ref{ex3d}.}\label{fig3d}
\end{figure}

\begin{corollary}\label{somult}
 Let $g\colon M\to\R^k$ and $f\colon\R\X M\to\R^k$ be tangent to the closed 
boundaryless submanifold $M$ of $\R^k$. Assume that $g$ is $C^2$, $g^{-1}(0)$ is 
compact and $f$ is $T$\hbox{-}periodic in $t$. Assume also that
\begin{enumerate}\renewcommand{\theenumi}{\roman{enumi}}
\item There are $n-1$ isolated zeros of $g$, $p_1,\ldots,p_{n-1}$ for which
the conditions of Theorem \ref{condisofin} hold and such that
\begin{equation*}
\sum_{i=1}^{n-1}\idx(g,p_i) \neq \deg(g,M) ;
\end{equation*}
\item The unperturbed equation
\begin{equation*}
\dot x =g(x)
\end{equation*}
does not admit unbounded connected sets of $T$\hbox{-}periodic solutions in $C_T(M)$.
\end{enumerate}
Then, for $\lambda >0$ sufficiently small, Equation \eqref{eq0} admits at least 
$n$ geometrically distinct $T$\hbox{-}periodic solutions.
\end{corollary}

\begin{remark}
For the purpose of analyzing the notion of non-$T$-resonance we might have as well used an open 
set of $\R$ and $\R^2$. We have not done so, partly, for the sake of simplicity but, more 
importantly, because multiplicity results like the ones discussed in Section \ref{secMul} require 
manifolds which are closed subsets of $\R^k$.
\end{remark}


\end{document}